\documentclass[12pt]{article}

\usepackage{amsmath}
\usepackage{amsfonts}
\usepackage{amsthm}
\usepackage{amssymb}
\usepackage{graphicx}

\newtheorem{theorem}{Theorem}[section]

\newtheorem{lemma}[theorem]{Lemma}

\theoremstyle{definition}

\title
{Gr\"unbaum Colorings of Toroidal Triangulations}
\author{Michael O. Albertson\thanks {Department of Mathematics and Statistics, Smith College Northampton, MA 01063 USA}, Hannah Alpert\thanks {340 Fox Drive,
Boulder, CO 80303 USA},\\  sarah-marie belcastro$^*$\thanks{corresponding author: email smbelcas@toroidalsnark.net, phone 413-341-1373, fax 413-585-3786}, and Ruth Haas$^*$}

\begin{document}
\maketitle



\begin{abstract} We prove that if $G$ is a triangulation of the torus and $\chi(G) \neq 5$, then 
there is a $3$-coloring of the edges of $G$ so that the edges bounding every face are assigned three different colors.

\end{abstract}



\section{Introduction}\label{sec:intro}

Our story begins with Kempe's famous (but false) proof of the Four Color Conjecture (4CC).  Subsequently Tait claimed, but did not publish, another proof based on the (false) belief that every cubic $3$-connected planar graph contains a Hamilton cycle.  Tait did understand that a $4$-coloring of the faces of such a graph is equivalent to a $3$-coloring of its edges.  See \cite{blw} for details of this fascinating story.  If one dualizes Tait's observation, one obtains an equivalent version of the 4CC for planar graphs, one with a different notion of edge coloring. We say that a $3$-coloring of the edges of an embedded triangulation is a {\it Gr\"unbaum coloring}
 if every facial triangle is incident with three different colors.  Note this places no constraint on the colors assigned to the edges of either a separating or a non-contractible $3$-cycle, nor does it require incident edges to receive different colors.  In 1968 Gr\"unbaum offered a far-reaching
  generalization of the 4CC: every triangulation of every orientable surface has a Gr\"unbaum coloring \cite{gc}.   Since the first proof of the Four Color Theorem (4CT) \cite{ah1, ah2, rsst}, Gr\"unbaum's conjecture has gained increasing notoriety \cite{arch}.

The purpose of this paper is to prove the following.

\begin{theorem} \label{mainthm}   If $G$ is a triangulation of the torus with $\chi(G) \neq 5$, then  $G$ has a Gr\"unbaum coloring.
\end{theorem}

Recently Kochol discovered counterexamples to Gr\"unbaum's conjecture on every orientable surface $S_g$ with
 $g \geq 5$ \cite{k}.
  Our result contrasts with Kochol's.   The following problems remain open:  Does every $5$-chromatic toroidal triangulation have a Gr\"unbaum coloring?  For $2\leq g \leq4$, does every triangulation of $S_g$ have a Gr\"unbaum coloring?  Does every {\it locally planar} triangulation of $S_g$ (those embedded with large edge width) have a Gr\"unbaum coloring?
  
Our proof is a divide-and-conquer
 argument and consists of a sequence of lemmas. In the second section we show that if $\chi(G) \leq 4, \chi(G) = 7$, or $G$ is $6$-regular, then $G$ has a Gr\"unbaum coloring.  In the third section we treat $6$-chromatic graphs whose critical $6$-chromatic subgraphs are not $K_6$.  Finally,  in the fourth section we complete the proof by examining those $6$-chromatic toroidal triangulations that contain $K_6$.   

\section{First Results }\label{sec:easy}

Our first result is folklore.  Many who think about Gr\"unbaum's conjecture would dualize the statement and reproduce Tait's proof.  We prefer a direct argument inspired by graph homomorphisms.

\begin{lemma} \label{lem4}  If $G$ is a triangulation of any surface and $\chi(G) \leq 4$, then $G$ has a Gr\"unbaum coloring.
\end{lemma}

\begin{proof} Fix a vertex $4$-coloring of $G$ and a standard edge 
$3$-coloring of $K_4$.  Fix a
vertex 4-coloring of $K_4$ by using the labels of the vertices as colors.
  Suppose $e = uv \in E(G), {\rm col}(u) = i$, and ${\rm col}(v) = j$;  color $e$ with col$(ij)$.
 Any $3$-cycle in $G$ corresponds to a triangle in $K_4$.  Since the edge coloring of $K_4$ is Gr\"unbaum, so is the edge coloring of $G$.
\end{proof}

\begin{lemma}\label{lem6reg}  Every $6$-regular toroidal graph has a Gr\"unbaum coloring.
\end{lemma}

\begin{proof}  It is an immediate consequence of Euler's formula that a $6$-regular toroidal graph is a triangulation.  Altshuler has shown that every $6$-regular triangulation of the torus can be realized as a rectangular grid with diagonals in which the left and right sides are identified and the top and bottom sides are identified with a twist\,\cite{a}.  Every face in such a graph has one vertical edge, one horizontal edge, and one diagonal edge. The result follows immediately upon realizing that \lq\lq vertical", \lq\lq horizontal", and \lq\lq diagonal" are colors. 
\end{proof}

\begin{lemma}\label{lemext}   Suppose $H$ is a triangulation of a surface $S$ that has a Gr\"unbaum coloring.  If $G$ is a triangulation of $S$ that contains $H$ as a subgraph, i.e.
  $G$ is a refinement of $H$, then $G$ has a Gr\"unbaum coloring.
\end{lemma}

\begin{proof}  Pick a Gr\"unbaum coloring of $H$.  Transfer this edge coloring to $G$.  The edges of $G$ that are not yet colored live inside triangular faces of $H$.  For each such triangular face of $H$, consider the triangulation of the plane generated by the vertices on and inside this face.  Since this is a planar triangulation, it has a Gr\"unbaum coloring by the 4CT.  A permutation of the colors in this Gr\"unbaum coloring will make it agree with the Gr\"unbaum coloring of $H$.  Since this can occur independently for every face in $H$, $G$ will have a Gr\"unbaum coloring.
\end{proof}

\begin{lemma}\label{lemseven}  Every 7-chromatic toroidal triangulation has a Gr\"unbaum coloring.
\end{lemma}

\begin{proof} Any toroidal graph with chromatic number 7 must contain $K_7$ as a subgraph \cite{d}. 
Since $K_7$ is $6$-regular,  Lemmas  \ref{lem6reg} and \ref{lemext} finish the proof.
\end{proof}

\section{Six-Chromatic Toroidal Triangulations, I}\label{sec:res1}

The proof that $6$-chromatic triangulations of the torus have Gr\"unbaum colorings is, in outline, much like the proof of Lemma \ref{lemseven}.  However, the details are tricky and will occupy the remainder of this paper.  Our proof relies on the fact that critical $6$-chromatic toroidal graphs are classified.  Clearly $K_6$ is one such graph.  The others are described below.

Given graphs $G_1$ and $G_2$, the {\it join} of $G_1$ and $G_2$, denoted by $G_1 + G_2$, is the graph obtained by taking vertex-disjoint copies of $G_1$ and $G_2$ and adding every edge $uv$ where $u \in G_1$ and $v \in G_2$.  It is straightforward to see that if $G_1$ is a critical $r$-chromatic graph and $G_2$ is a critical $s$-chromatic graph, then $G_1 + G_2$ is a critical $(r+s)$-chromatic graph.
Thus $C_3 + C_5$ is a critical $6$-chromatic graph.

Let $H_7$ denote the graph obtained by applying the first step of the Haj${\rm \acute{o}}$s construction \cite{JT} to two vertex-disjoint copies of $K_4$. Since the Haj${\rm \acute{o}}$s construction preserves criticality, $H_7$ is a critical $4$-chromatic graph.  Consequently,
 $H_7 + K_2$ is a critical $6$-chromatic graph.  $H_7$ is shown in Figure \ref{figH7}.

\begin{figure} [h]
\begin{center} 
\includegraphics[scale=.5]{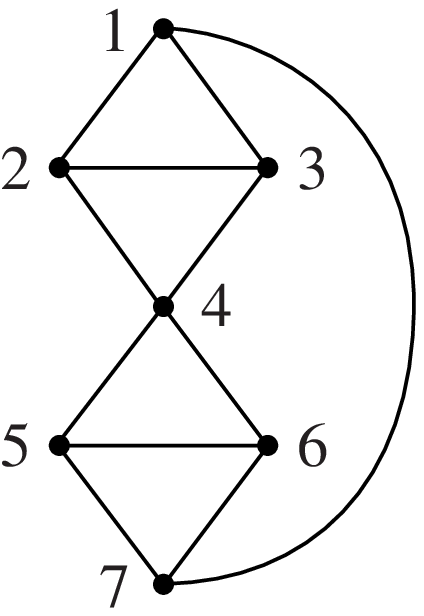}
\end{center}
\caption{$H_7$}
\label{figH7}
\end{figure}

More than thirty years ago, Albertson and Hutchinson \cite{ah4, ah3} discovered a $6$-regular, critical $6$-chromatic toroidal graph and named it $J$.  It was later realized that $J$ is isomorphic to $C_{11}^3$, the cube of the $11$-cycle.  In \cite{ah5} Albertson and Hutchinson gave a list of critical $6$-chromatic graphs including $K_6, C_3 + C_5$, and $H_7 + K_2$, but mistakenly believed the list was missing an additional critical graph on nine vertices.
Fourteen years later Thomassen proved the following.

\begin{theorem}[\cite{thom94,MT}]  If $G$ is a $6$-chromatic toroidal graph, then $G$ contains 
exactly one of $K_6, C_3 + C_5, H_7 + K_2$, or $C_{11}^3$.
\end{theorem}

Given $G$, a $6$-chromatic triangulation of the torus,  we will produce an edge coloring of its critical $6$-chromatic subgraph that extends to a Gr\"unbaum coloring of all of $G$.  If a face of a particular embedding of the critical $6$-chromatic subgraph is $3$-sided, and the three boundary edges of this face are assigned different colors, then the interior of this triangular face has a Gr\"unbaum coloring of its interior.   The argument is identical to the proof of Lemma  \ref{lemext}.

\begin{lemma}  If $G$ is a $6$-chromatic toroidal graph that contains
 $C_{11}^3$, then $G$ has a 
Gr\"unbaum coloring.
\end{lemma}

\begin{proof}  This follows immediately from Lemmas \ref{lem6reg} and \ref{lemext}.  
\end{proof}

If $G$ is a $6$-chromatic toroidal graph that does not contain $C_{11}^3$, the proof that $G$ has a Gr\"unbaum coloring is complicated by the fact that the critical $6$-chromatic subgraph does not embed as a triangulation of the torus.  


\begin{lemma}\label{H7embed} $H_7 + K_2$ is uniquely embeddable in the torus.
\end{lemma}

\begin{proof}  Gagarin 
 et.\,al.\,have shown that all embeddings of a graph $G$  on the torus can be found by considering a spanning $\theta$-subgraph of $G$ \cite{Kocay}. (A $\theta$-graph
  is composed of two vertices of degree 3 joined by three internally disjoint paths.) 
There are three labeled embeddings of a $\theta$-graph on the plane, each with a different path in the middle. Similarly, there are three labeled embeddings of a $\theta$-graph on a cylinder. There is only one labeled 2-cell embedding of a $\theta$-graph on the torus. These seven cases correspond to possible
 toroidal embeddings of a spanning $\theta$-subgraph of $G$.  In each case, the remaining edges of $G$ can be added one by one as forced by the embedding or to make appropriately sized faces. In all but one case, this process leads to a contradiction. 


Denote the vertices of $H_7$ by $\{1, 2, \dots  , 7\}$ as in Figure \ref{figH7} and the vertices of $K_2$ by $\{a, b\}$. Choose the spanning $\theta$-subgraph $G_\theta$ to have $\{a, b\}$ as its vertices of degree 3.  Let the three paths be $a4b$, $a213b$, and $a 5 7 6 b$. Euler's formula implies that one face will be a quadrilateral and the rest triangles.  We provide a few examples of the process of ruling out potential embeddings.

{\bf Case I.} Embed $G_\theta$ in the plane with $a576b$ as the middle path. The placement of the edges $45$, $46$ and $17$ is now forced. The edge $5b$ cannot be placed without crossings.

{\bf Case II.} Embed $G_\theta$ on the cylinder with paths $a213b$ and $a 4b$ homotopic. Consider the possible placements of the edge $ab$.  If it is also placed homotopic to the paths $a21 3 b$ and $a 4b$ then either one of the edges $17$ or  $46$ is immediately blocked or the edges $24$ and $34$ are forced and the edge $17$ cannot be placed. If $ab$ is not homotopic, then there are 2 other cases and several subcases, all of which lead to contradictions.

{\bf Case III.} Embed $G_\theta$ in the plane with $a4 b$ as the middle path. The  placement of the edges incident to $4$ is forced. This produces two quadrilateral faces, $(1, 3, 4, 2)$ and $(4, 5, 7, 6)$. There is only one quadrilateral face in the embedding of $G$, so at least one of the edges $23$ or $56$ must be placed in a quadrilateral. Without loss of generality, assume $23$ divides the quadrilateral $(1, 3, 4, 2)$. The remaining edges are now added to create triangular faces. This produces the unique embedding.  The other cases are  more involved but follow similar strategy. \end{proof}

\begin{lemma} \label{TorusH7} If $G$ is a triangulation of the torus that contains $H_7 + K_2$, then $G$ has a 
Gr\"unbaum coloring.
\end{lemma}

\begin{proof}  The unique embedding of $H_7 + K_2$ contains one quadrilateral face, say $\Gamma$.  We construct a triangulation of the plane by cutting $\Gamma$ and its interior out of the torus and placing it in the plane.  We then create a new vertex, say $u$, in the unbounded face and add edges from $u$ to each vertex of $\Gamma$.  The resulting graph $\Gamma^*$ is a triangulation of the plane and consequently has a Gr\"unbaum coloring.  In such a coloring the edges incident with $u$ can use either two or three colors.  If the edges incident with $u$ use only two colors, then these colors must alternate.  Consequently, the colors on the edges of $\Gamma$ must all be the same.  If there are three colors used on the edges incident with $u$, then one color must be repeated on opposite edges, and the other two colors used on one edge each.  In this case the colors on the edges of $\Gamma$ must occur in consecutive pairs.  This can happen in only two different ways up to color permutation.   A priori we don't know which of these colorings of the edges incident with $u$ and bounding $\Gamma$ extend to the interior of $\Gamma^*$.   At least one such coloring must extend since $\Gamma^*$ has a Gr\"unbaum coloring.  We finish the proof by exhibiting in Figure \ref{k2h7c3c5col} three different edge colorings of $H_7 + K_2$,  one for each possible edge coloring of the edges in $\Gamma$.  \end{proof}

\begin{lemma} $C_3 + C_5$ is uniquely embeddable in the torus.
\end{lemma}

\begin{proof} Proceeding as in Lemma \ref{H7embed} we consider a spanning $\theta$-subgraph of $C_3 + C_5$.
Denote the vertices of $C_3$ by $\{a, b, c\}$ and those of $C_5$ by $\{1, 2, 3, 4, 5\}$.
Choose the spanning $\theta$-subgraph $G_\theta$ to have the vertices  $\{ a, c\}$ and the 
three paths $ab c$,  $a 1 23c$, $a 5 4 c$.  We will again have exactly one quadrilateral face and the rest triangles, and again we provide sample cases.

{\bf Cases I, II.} The two potential planar embeddings with either path of $C_5$ as the middle path of $G_\theta$ each induce a 5-sided face and are therefore impossible. 

{\bf Case III.} The 2-cell embedding of $G_\theta$ on the torus extends uniquely to an embedding of  $C_3 + C_5$.\end{proof}

\begin{lemma} If $G$ is a triangulation of the torus that contains $C_3 + C_5$, then $G$ has a 
Gr\"unbaum coloring.
\end{lemma}

\begin{proof}  Since $C_3 + C_5$ has a unique embedding with exactly one non-triangular face that happens to be $4$-sided,  the proof follows that of Lemma \ref{TorusH7}.  The three different edge colorings  of $C_3 + C_5$ that we need to complete the proof are  presented in Figure \ref{k2h7c3c5col}.
\end{proof}

 \begin{figure}[h]
\begin{center}
\includegraphics[scale=.3]{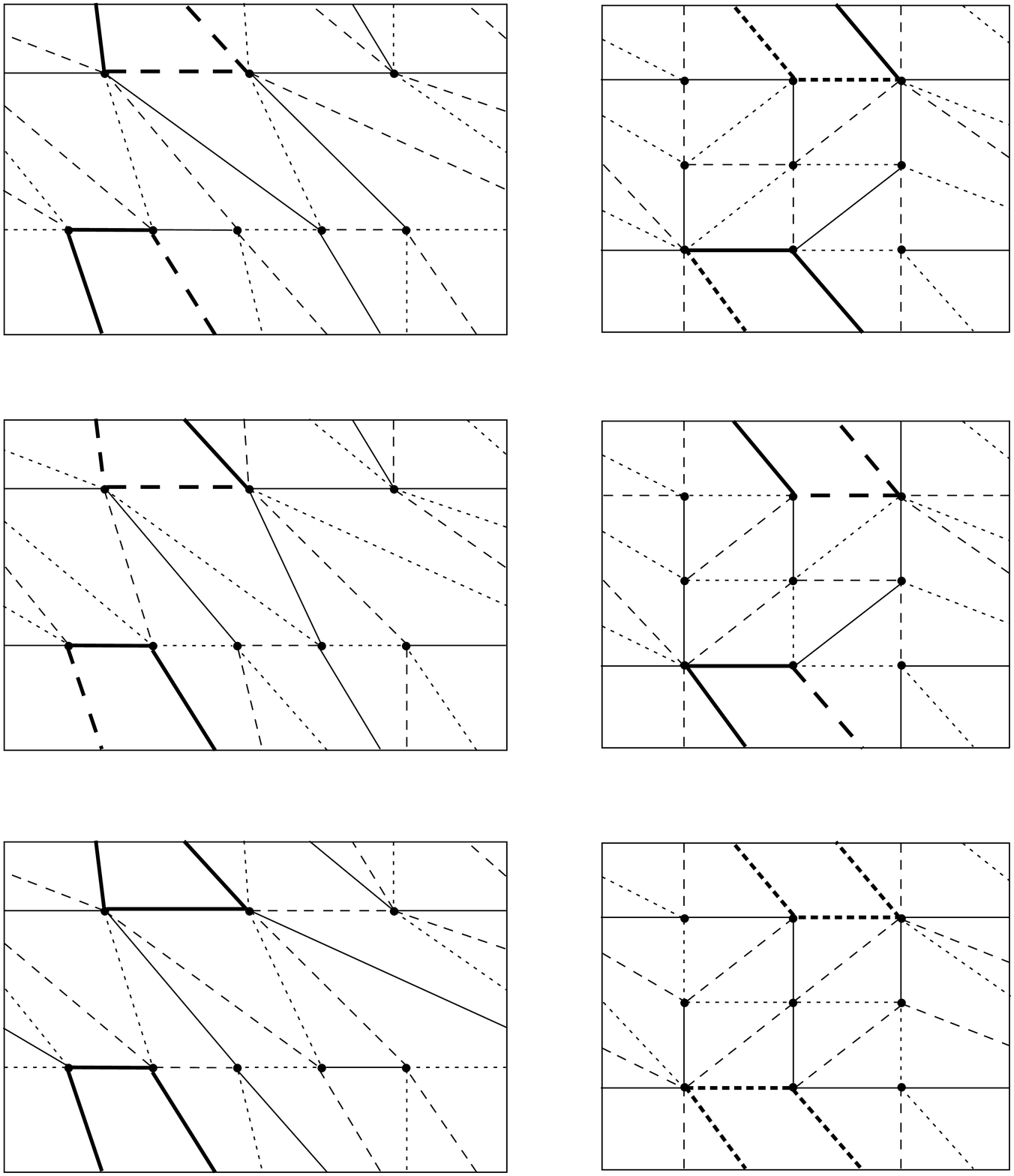}
\caption{ Colorings of $C_3+C_5$ and $H_7 + K_2$.
}
\label{k2h7c3c5col}
\end{center}
\end{figure}

\section{Six-Chromatic Toroidal Triangulations, II}\label{sec:res2}

$K_6$ has 4 embeddings (three shown in Figure \ref{k6es}, and the last in Figure \ref{6col}), which are identified by their non-triangular faces: $(4,4,4)_A$ has three squares forming a laddered cylinder, $(4,4,4)_B$ has three squares forming a non-laddered cylinder, $(5,4)$ has one pentagonal and one square face, and $(6)$ has one hexagon. 


\begin{figure}[h]
\begin{center}
\includegraphics[scale=.5]{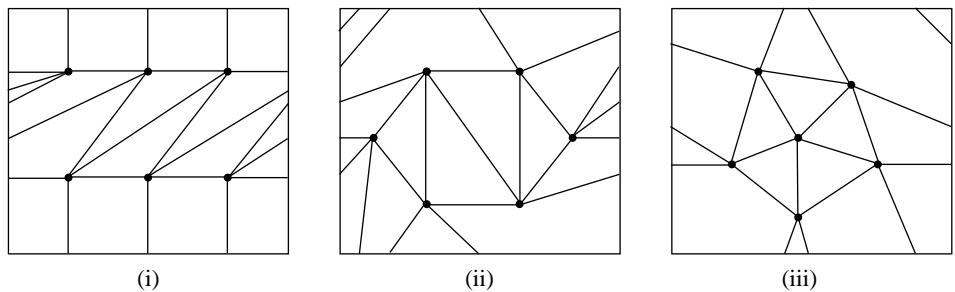}
\caption{Three toroidal embeddings of $K_6$: (i) $(4,4,4)_A$, (ii) $(4,4,4)_B$, and (iii) $(5,4)$.}
\label{k6es}
\end{center}
\end{figure}



Suppose $G$ is an embedded graph.  A $3$-coloring of a subset of the edges of $G$ is called a {\it partial Gr\"unbaum coloring} if the edges of every facial triangle are assigned three distinct colors.
Note that if we have a (partial) Gr\"unbaum coloring of an embedded graph $G$, then we can use Kempe chain arguments to produce alternative (partial) Gr\"unbaum colorings.  
Formally, we consider $G^\ast$, the dual embedding of $G$ on the torus.  Note that $G^\ast$ is a cubic graph.  If a particular edge, say $e$, is colored $p$, then a $p-t$ {\it Kempe chain} at $e$ is the set of edges that are colored either $p$ or $t$ and correspond with the standard edge Kempe chain in $G^\ast$ containing $e$.  A $p-t$ {\it Kempe change} at $e$ switches the colors in the $p-t$ Kempe chain at $e$. 


\begin{lemma}[\cite{eh}] \label{milesnate1} \rm 
  Suppose $\Gamma$ is a separating square in $G$, a  triangulation of the plane.  In any Gr\"unbaum coloring of $G$ each color appears an even number of times on $\Gamma$.  
\end{lemma}

\begin{proof} In any Gr\"unbaum coloring of a triangulation of any surface every Kempe chain is a cycle.
   Thus each edge of $\Gamma$ is in a Kempe chain with some other edge of $\Gamma$.  Since there are three colors used on four edges of $\Gamma$, at least one color must appear an even number of times.  Call this color $t$.  Suppose color $p$ appears an odd number of times on $\Gamma$.  Consider the $t-p$ Kempe chains.  They are disjoint, and each contains an even number of edges of $\Gamma$.  However, the total number of edges in $\Gamma$ colored either $t$ or $p$ is odd, a contradiction. 
\end{proof}

\noindent The preceding argument can be easily modified to give the following.

\begin{lemma} \label{hannahlem}  Suppose $\Gamma$ is a separating $2n$-gon (resp.\,$(2n+1)$-gon)
 in $G$, a  triangulation of the plane.  In any Gr\"unbaum coloring of $G$ each color appears an even (resp.\,odd) number of times on $\Gamma$.  
\end{lemma}

It follows that any separating square in a Gr\"unbaum-colored triangulation may have only the following types of colorings: $tttt$ (denoted C for ``constant"), $tptp$ (denoted A for ``alternating"), $ttpp$ (denoted B$_1$), and $tppt$ (denoted B$_2$).  Note that C and A are invariant under rotations and recoloring.  While B$_1$ and B$_2$ are equivalent under rotation, the orientation of a square within an embedding may require one and exclude the other.

 \begin{lemma}[\cite{eh}]\label{lemmilesnate2} 
  In the plane, a triangulation of the interior of a square has a partial Gr\"unbaum coloring of type  B$_1$ or  B$_2$ if and only if it also has a partial  Gr\"unbaum coloring of type A or C.
\end{lemma}

\begin{proof}  We give one direction of one case of the proof.  The other direction follows immediately since Kempe changes are reversible.  The second case is identical.  Suppose we have a triangulation of the interior of a square $\Gamma$.  Assume $\Gamma$ has a partial  Gr\"unbaum coloring of type B$_1$.  Consider the two  segments of 
 $p-t$ Kempe chains intersecting  the edges of $\Gamma$.  Either one of these intersects $\Gamma$ in two edges colored $t$ and the other intersects $\Gamma$ in two edges colored $p$, or  there are two Kempe chains each intersecting $\Gamma$ in one edge colored $t$ and one edge colored $p$.  In the former case we can make a Kempe change to yield a partial  Gr\"unbaum coloring of type C.  In the latter case where the edges of  $\Gamma$ are colored $ttpp$ we claim that Kempe chains must join consecutive edges; if not,
 there will be two crossing Kempe chains, contradicting the Jordan Curve Theorem.
    Then,  we can make one Kempe change to create a partial  Gr\"unbaum coloring of type A.
\end{proof}

\subsection{The $(4,4,4)$ embeddings of $K_6$.}

\begin{lemma}\label{lem444b}  Every $6$-chromatic toroidal triangulation containing the $(4,4,4)_B$ embedding of $K_6$ has a Gr\"unbaum coloring.
\end{lemma}

\begin{proof} Consider the three squares of the $(4,4,4)_B$ embedding.  We enumerate the colorings of   this configuration, and show that each possibility is compatible with a partial  Gr\"unbaum coloring of $(4,4,4)_B$.  Consider three specific exterior triangulations of the square, the 4-wheel and the two ways of adding a diagonal.  The first exterior triangulation gives rise to colorings of the square C, B$_1$, or B$_2$; one of the added-diagonal exterior triangulations has colorings of the square A or B$_1$; and the other added-diagonal exterior triangulation has colorings of the square A or B$_2$.  Therefore any triangulated square on a disk may be colored as (A or B$_1$) and (A or B$_2$) and (C or B$_1$ or B$_2$).


From Lemma \ref{lemmilesnate2}, we know that any square may be colored as (B$_1$ or B$_2$) and (A or C).  These conditions combine to give the logically equivalent statement that any square on a disk may be colored (A and B$_1$) or (A and B$_2$) or (C and B$_1$ and B$_2$).  Calling the coloring possibilities  (A and B$_1$) type 1, (A and B$_2$) type 2, and (C and B$_1$ and B$_2$) type 3, we have 27 total possibilities for coloring the set of three squares. 
Note that 111, 112, 121, 211, 221, 212, 122, and 222 may all be colored A A A; 113, 123, 311, 313, and 321 may all be colored B$_1$ A B$_1$; 223, 233, 323, and 333 may all be colored B$_2$ B$_2$ C;  213, 131, 133, and 231 may be colored A B$_1$ B$_1$; 331, 332, 132, and 312 may all be colored B$_1$ B$_1$ A; and, 232 and 322 may be colored B$_2$ B$_2$ A.
  It is now straightforward to verify that each of these colorings is compatible with a partial Gr\"unbaum coloring of 
 $(4,4,4)_B$ as shown in Figure \ref{444bcol}.
 \end{proof}
 
 \begin{figure}[h]
\begin{center}
\includegraphics[scale=.5]{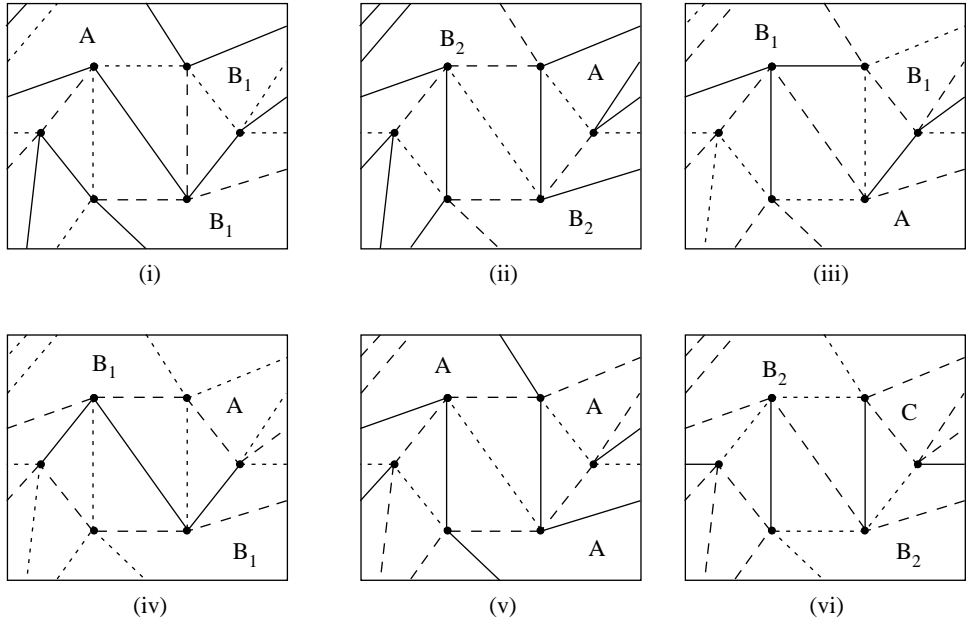}
\caption{Six  partial Gr\"unbaum colorings of $(4,4,4)_B$: (i) A B$_1$ B$_1$, (ii) B$_2$ B$_2$ A, (iii) B$_1$ A B$_1$, (iv) B$_1$ B$_1$ A, (v)  A A A, and (vi) B$_2$ B$_2$ C.}
\label{444bcol}
\end{center}
\end{figure}

\begin{lemma}\label{lem444a} Every $6$-chromatic toroidal triangulation containing the $(4,4,4)_A$ embedding of $K_6$ has a Gr\"unbaum coloring.
\end{lemma}

\begin{proof} Using the same notation as in the preceding lemma, we have 27 total possibilities for coloring the triple of squares.  However, the embedding $(4,4,4)_A$ is rotationally symmetric along the cylinder formed by the triple of squares, so these reduce to 11 possibilities.  Note that 113, 133, 233, 231, and 213 may be colored A B$_1$ B$_1$; 223 may be colored A B$_2$ B$_2$; 111, 112, 122, and 222 may be colored A A A; and, 333 may be colored C C C.  It is straightforward to verify that each of these colorings is compatible with a partial Gr\"unbaum coloring of  $(4,4,4)_A$ as shown in Figure \ref{444acol}.
 \end{proof}
 
 \begin{figure}[h]
\begin{center}
\includegraphics[scale=.55]{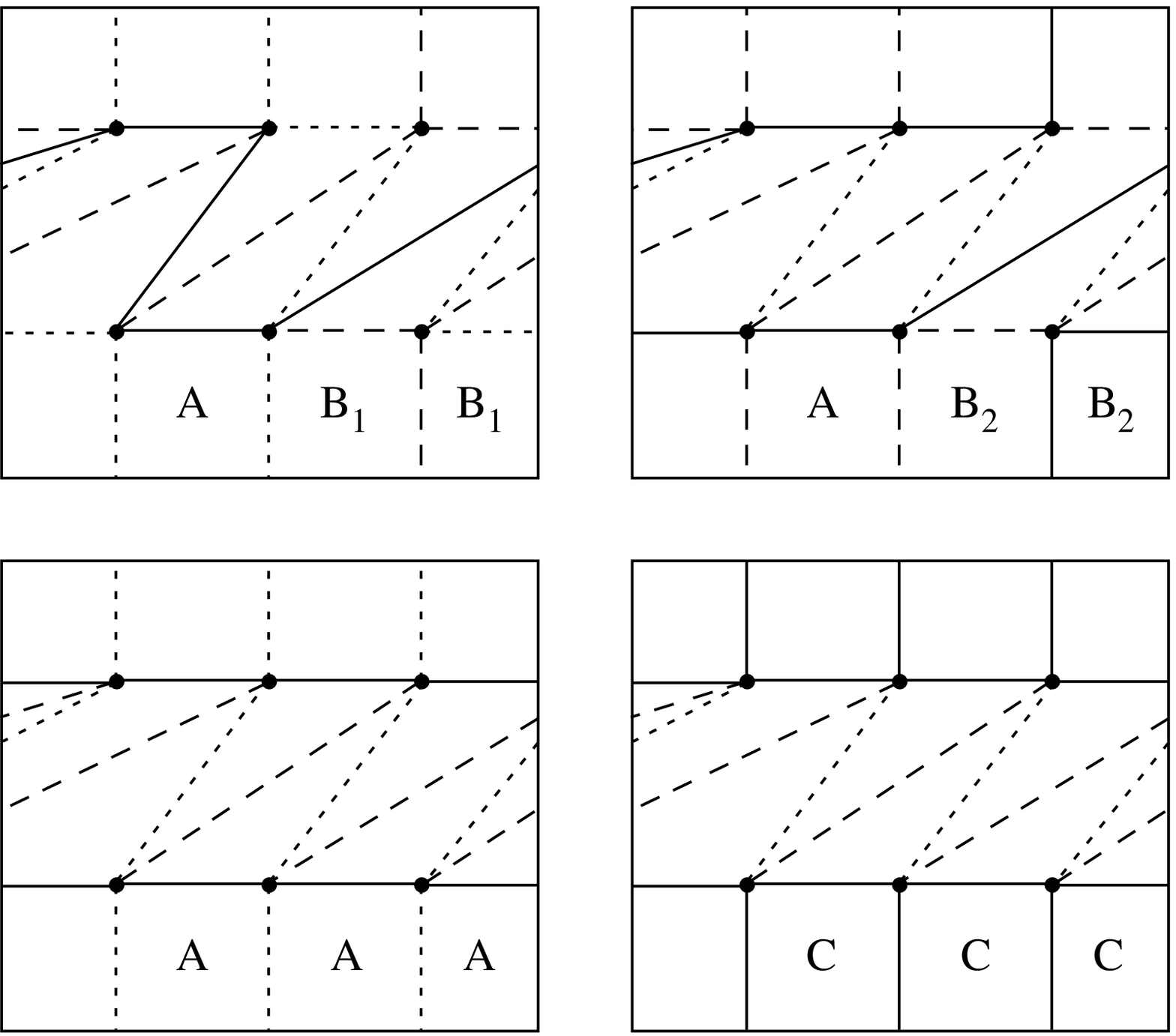}
\caption{Four partial Gr\"unbaum colorings of $(4,4,4)_A$.}
\label{444acol}
\end{center}
\end{figure}

\subsection{The $(5,4)$ embedding of $K_6$.}

\begin{lemma}\label{lem54} Every $6$-chromatic toroidal triangulation containing the $(5,4)$ embedding of $K_6$ has a Gr\"unbaum coloring.
\end{lemma}

 \begin{figure}[h!]
\begin{center}\includegraphics[scale=.55]{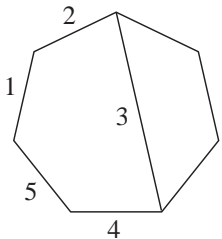}
\caption{The non-triangular faces of the $(5,4)$ embedding of $K_6$.}
\label{54faces}
\end{center}
\end{figure}

\begin{proof}
Consider a triangulation of the torus containing the $(5,4)$ embedding of $K_6$.  We label the edges of the triangulated pentagon as shown in Figure \ref{54faces}.  Removing the triangulated pentagon, placing it in the plane, and triangulating its exterior, we obtain a Gr\"unbaum coloring of this triangulation and thus a partial Gr\"unbaum coloring of the triangulated pentagon.  We know from Lemma \ref{hannahlem} that each of the three colors must appear an odd number of times.  Thus one color, say $t$, must appear exactly three times while the other two colors, say $p$ and $g$, 
   each appear once.   We let $j;k$ denote the instance where $p$ is assigned to edge $j$ and $g$ is assigned to edge $k$. Since the naming of the colors could be reversed, if a coloring is of the form $j;k$, it could also be labeled $k;j$.  


Suppose we have a triangulated pentagon in the disk colored so that the three edges colored $t$ are consecutive, i.e.,
 $|k-j| = 1$.  As in the proof of Lemma \ref{lemmilesnate2},
  the $j;j+1$ coloring can be Kempe changed to either $j;j+2$ or $j;j+4$ but not to $j;j+3$.  Reversing colors, we see that a triangulated pentagon colored $j;j+4$ may be Kempe changed to one colored  $j;j+1$ or $j;j+3$. Similarly, a triangulated pentagon colored $j;j+2$ may be Kempe changed to one colored $j;j+1$ or $j;j+3$ and a triangulated pentagon colored $j;j+3$ may be Kempe changed to one colored $j;j+2$ or $j;j+4$. 


Now suppose we have a triangulation of the torus containing the $(5,4)$ embedding of $K_6$.  It contains a triangulated pentagon whose colorings we have discussed above,
 and a triangulated square that has
  coloring type 1 (A and B$_1$), type 2 (A and B$_2$), or type 3 (C and B$_1$ and B$_2$).  Examining the partial Gr\"unbaum colorings of $(5,4)$ in Figure \ref{54col}, we see that these can be colored with corresponding pentagonal and square colorings, as 
$tptpptg \rightarrow 2;5$ B$_1$,
$tpgpptt \rightarrow 2;3$ B$_1$,
$tttppgp \rightarrow 4;5$ B$_1$,
$tpptpgt \rightarrow 2;4$ A,
$ tppgpgg \rightarrow 1;2$ A, and
$ttptppg \rightarrow 4;5$ A.

 \begin{figure}[h]
\begin{center}
\includegraphics[scale=.5]{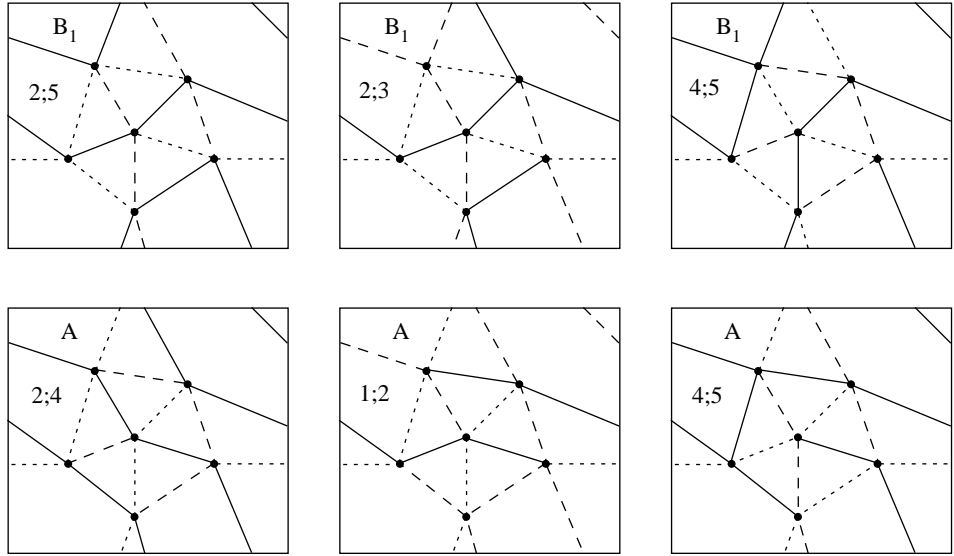}
\caption{The six required partial Gr\"unbaum colorings of $(5,4)$.}
\label{54col}
\end{center}
\end{figure}

We now argue that any possible triangulation of a pentagon identified along an edge with a square has a partial Gr\"unbaum coloring that
 can be Kempe changed to one of the six partial Gr\"unbaum colorings listed above.  First, suppose that the square has coloring type 1 or 2.  It may be colored A, and so a pentagon with a triangulation-induced coloring $2;4$, $1;2$, or $4;5$ will (together with an A-colored square) form a heptagon compatible with a partial Gr\"unbaum coloring of $(5,4)$.  We proceed with the remaining 7 $j;k$ colorings of the pentagon, assuming the square is type 1 or 2,  by  sequentially reducing them  to cases already accounted for.
 A pentagon with coloring $4;1$ can be Kempe changed to be colored $4;5$ or $4;2$, both of which are already known to (together with a type 1 or type 2 square) form a heptagon compatible with a partial Gr\"unbaum coloring of $(5,4)$.  Similarly  a pentagon with coloring $4;3$ can also be Kempe changed to $4;5$ or $4;2$.  Pentagons with coloring $1;3$ or  with coloring $1;5$  can be Kempe changed to be colored $1;2$ or $1;4$, both of which are already known to form compatible heptagons.  Those with coloring $2;3$ or $2;5$ can be Kempe changed to be colored $2;4$ or $2;1$, both of which are already done.   Finally, a pentagon with coloring $3;5$ can be Kempe changed to be colored $3;4$ or $3;1$, both of which are already done.  This accounts for all ten possible partial Gr\"unbaum colorings of a pentagon combined with a type 1 or type 2 square to form a heptagon.


Now suppose that the square has coloring type 3.  It may be colored B$_1$, and so a pentagon with triangulation-induced coloring $2;5$, $2;3$, or $4;5$ will (together with a B$_1$-colored square) form a heptagon compatible with a partial Gr\"unbaum coloring of $(5,4)$. We follow the same type of strategy as above.
Pentagons with coloring $2;4$ or   $2,1$ can each be Kempe changed to one of $2;3$ or $2;5$. Those with $5;1$ or $5,3$ can  be changed to one of $5,2$ and $5;4$.  Next, those with $4;1$ or $4;3$  each change to one of $4;5$ and $4;2$.
 Finally, a pentagon with coloring $1;3$ can be changed to either $1;2$ or $1;4$. 
 This accounts for all ten possible partial Gr\"unbaum colorings of a pentagon combined with a type 3 square to form a heptagon, and completes the proof.

\end{proof}

\subsection{The $(6)$ embedding of $K_6$.}

\begin{lemma}\label{lem6} Every 6-chromatic toroidal triangulation containing the $(6)$ embedding of $K_6$ has a Gr\"unbaum coloring.
\end{lemma}

\begin{proof} Consider a triangulation of the torus containing the $(6)$ embedding of $K_6$.  The hexagonal face of $(6)$ is a triangulated hexagon; this may be removed, placed in the plane, and by triangulating the exterior we see that the resulting graph has a Gr\"unbaum coloring and thus the triangulated hexagon has a partial Gr\"unbaum coloring.  By Lemma \ref{hannahlem},  
the number of hexagon edges in each color must be even.  These may be allocated to $g,p$, and $t$ as 
2, 2, 2, or 0, 2, 4, or 0, 0, 6.

Because $K_7$ has a unique embedding on the torus, and $(6)$ is merely $K_7$ with one vertex removed, $(6)$ is dihedrally symmetric outside of the hexagonal face.  Thus, we can consider partial Gr\"unbaum colorings up to rotation and reflection; the orientation of a triangulation inside the hexagonal face can determine the colors on the edges of the hexagon, which, if they match any rotation/reflection of a partial Gr\"unbaum coloring of $(6)$, will determine a coloring of the remainder of the triangulation.

Let us list all possible partial Gr\"unbaum colorings of a hexagon, up to rotations and reflections.  We have $pppppp$, $ttpppp$, $tptppp$, $tpptpp$, $ttppgg$, $ttpgpg$,  $tpgtgp$, $tpgtpg,$ and $tpptgg$.  Of these, $ttpppp, tpptpp, ttppgg,$ and $tpgtpg$ appear as part of partial Gr\"unbaum colorings of $(6)$, shown in Figure \ref{6col}.  It remains to be shown that if a partial Gr\"unbaum coloring of a triangulated hexagon induces any of the other five colorings, there exist Kempe changes that render the coloring equivalent to one of the four hexagonal subsets of partial Gr\"unbaum colorings of $(6)$.

 \begin{figure}[h]
\begin{center}
\includegraphics[scale=.6]{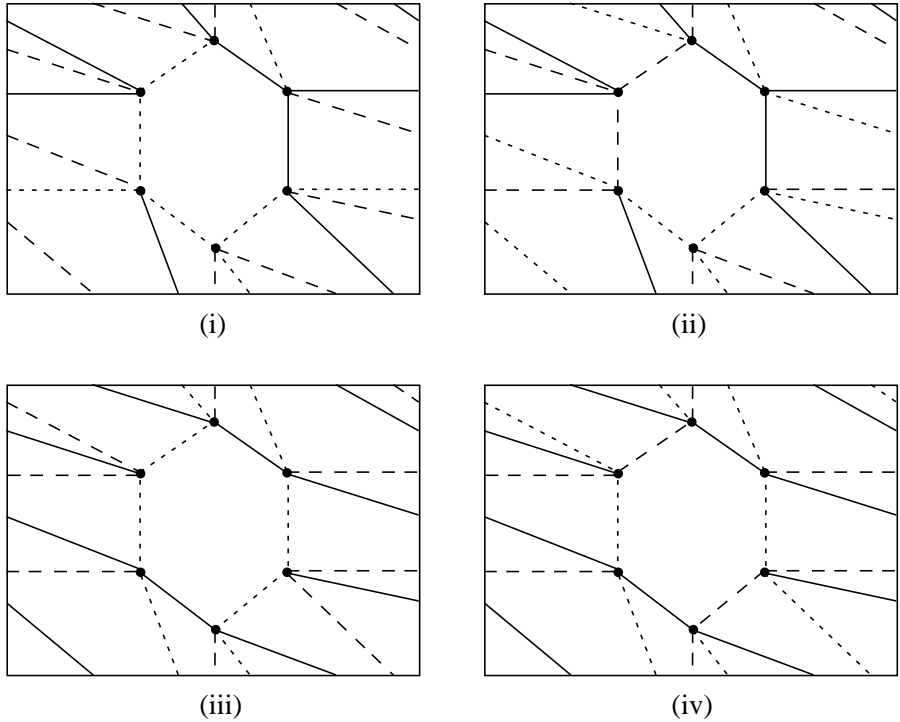}
\caption{Four partial Gr\"unbaum colorings of $(6)$, (i) $ttpppp$, (ii)  $ttppgg$, (iii) $tpptpp,$ and (iv) $tpgtpg$.}
\label{6col} 
\end{center}
\end{figure}

Now, note that if two colors use exactly 4 hexagon edges, occurring twice each or with all four edges of one color and none of the other, then these four edges behave as those of a square.  That is, the four edges are colored A, C, B$_1$, or B$_2$.  Thus, the proof of Lemma \ref{lemmilesnate2} holds, and we may conclude that Kempe chains connect only adjacent edges (where adjacency is considered relative to the ``square") and not opposing edges (where again, adjacency is considered relative to the ``square").

Consider each of the five remaining colorings in turn. For each make a 
Kempe change on the first $p$ edge. The same argument as Lemma \ref{lemmilesnate2} determines the possible new colorings.  These will either be, or color permute to, a coloring already determined  compatible with a partial Gr\"unbaum coloring of $(6)$.

If a partial Gr\"unbaum coloring of a triangulated hexagon induces the coloring  $tpgtgp$ on the hexagon, make a $p-g$ Kempe change on the first $p$ edge to obtain $tgptgp$, which color-permutes to $tpgtpg$,  or to obtain $tggtgg$ which color-permutes to $tpptpp$.
For the  induced hexagon  coloring $tpptgg$,  make a $p-g$ Kempe change on the first $p$ edge to obtain $tggtgg$, or to obtain  $tgptgp$, which color-permutes to $tpgtpg$.
For the  induced hexagon  coloring  $ttpgpg$,  make a $p-g$ Kempe change on the first $p$ edge to  obtain $ttgppg$, which color-permutes to $ggtppt$, which in turn reflects to $tpptgg$, or  obtain $ttggpp$, which color-permutes to $ttppgg$.
For the  induced hexagon  coloring  $tptppp$, make a $p-g$ Kempe change on the first $p$ edge to obtain $tgtgpp$, which color-permutes to $gpgptt$, which in turn reflects to $ttpgpg$, or  obtain $tgtppg$, which color-permutes to $gpgttp$ and rotates to $ttpgpg$.
Finally, for the  induced hexagon  coloring  $pppppp$,  make a $p-t$ Kempe change on the first $p$ edge to obtain one of $ttpppp, tptppp$, or $tpptpp$.    The colorings $ttpppp$ and $tpptpp$ are compatible with partial Gr\"unbaum colorings of $(6)$, and $tptppp$ was dealt with above.  
\end{proof}

\emph{Acknowledgments.}  Thanks to sarah-marie belcastro's topological graph theory class at the 2007 session of the Hampshire College Summer Studies in Mathematics (particularly Miles Edwards and Nate Harman), for enthusiastically working on open problems, some of whose solutions appear here.


\end{document}